\documentclass{amsart}

\usepackage{amsmath,amssymb,latexsym,amsthm,graphicx}

\newtheorem{theorem}{Theorem}[section]

\newtheorem{lemma}[theorem]{Lemma}

\newtheorem{coro}[theorem]{Corollary}

\newcommand{\mS}{\mathcal{S}}

\newcommand{\mA}{\mathcal{A}}

\newcommand{\mB}{\mathcal{B}}

\newcommand{\mE}{\mathcal{E}}

\newcommand{\mR}{\mathcal{R}}

\newcommand{\mP}{\mathcal{P}}

\newcommand{\mD}{\mathcal{D}}

\newcommand{\mK}{\mathcal{K}}

\newcommand{\rN}{\mathbb{R}}

\newcommand{\mT}{\mathcal{T}}

\newcommand{\mO}{\mathcal{O}}

\newcommand{\tT}{\mbox{T}}

\newcommand{\tN}{\mbox{N}}

\newcommand{\nN}{\mathbb{N}}

\newcommand{\uS}{\mathbb{S}}

\newcommand{\wf}{\mbox{WF}}

\newcommand{\eg}{\varepsilon}

\newcommand{\llg}{\lambda}

\newcommand{\sg}{\sigma}

\newcommand{\og}{\omega}

\newcommand{\Og}{\Omega}

\title[Oscillatory integral and spherical Radon transform]{Inverting Spherical Radon Transform by a Closed-form Formula: A Microlocal Analytic Point of View}

\author{Linh V. Nguyen}

\thanks{The research is supported by the NSF grant DMS 1212125}

\address{Department of Mathematics, University of Idaho, Moscow, Idaho 83844, USA}
\email{lnguyen@uidaho.edu}

\begin{document}

\maketitle

\begin{abstract} Let $\mR$ be the restriction of the spherical Radon transform to the set of spheres centered on a hypersurface $\mS$.  We study the inversion of $\mR$ by a closed-form formula. We approach the problem by studying an oscillatory integral, which depends on the observation surface $\mS$ as a parameter. We then derive various microlocal analytic properties of the associated closed-form inversion formula.  
\end{abstract}

\section{Introduction}\label{S:intro}
Let $\mS$ be a smooth hypersuface in $\rN^n$ and $f: \rN^n \to \rN$. We define the (restricted) spherical Radon transform $\mR(f)$ of $f$ by the formula: \begin{equation} \label{E:SM} \mR(f)(z,r) =\int\limits_{\uS_r(z)} f(y) \; d\sg(y),~(z,r) \in \mS \times \rN_+.\end{equation} Here, $\uS_r(z)\subset \rN^n$ is the sphere of radius $r$ centered at $z$ and $d\sg$ is the surface measure on $\uS_r(z)$. The transform $\mR$ plays an important role in thermo/photo-acoustic tomography (TAT/PAT) (see, e.g., \cite{FPR,FHR,KKun,kuchment10mathematics}). In TAT/PAT, $f$ is the image of the biological tissue of interest, which needs to be reconstructed. The function $\mR(f)$ is the available data, which is (roughly) the pressure wave recorded by the transducers located on the {\bf observation surface} $\mS$.  The main goal of TAT/PAT is to invert $\mR$, i.e., to find $f$ from $\mR(f)$.  This problem also appears in several other imaging modalities, such as ultrasound tomography (see, e.g., \cite{norton80reconstruction,norton80reconstruction2,norton81ultrasonic,norton79ultrasonic,AmAn}), SONAR (see, e.g., \cite{QuintoSONAR,louis00local}) and SAR (see, e.g., \cite{cheney00tomography,NoChe,SUSAR}). As a result, it has attracted a substantial amount of work. 

In this article, we are interested in inverting $\mR$ by closed-form formulas. Several such formulas have been found when $\mS$ is a sphere, cylinder, hyperplane, ellipse, and polygon with certain symmetries (e.g., \cite{FPR,XW05,FHR,Kun07,IPI,KKun,BuKar,NaRa,Pal-Uniform,natterer2012photo,Halt-Inv,Salman,KunPoly}). The obtained formulas look very different and, in many cases, they only coincide in the range of $\mR$ (see \cite{IPI} the their relations in the case of spherical surface $\mS$). 

Whether closed form inversion formulas exist for a general surface $\mS$ is still an open question. The approach by \cite{Pal-Uniform} gives an inversion formula up a compact operator. However, the nature of that compact operator is still not understood. The approach by \cite{natterer2012photo,Halt-Inv} gives an inversion formula up to a smoothing operator whose kernel was explicitly obtained.

An important scenario in imaging problems is the partial (limited) data phenomenon. That is, the data is only collected at a subset of the observation $\mS$ and the collecting time is finite. It is quite desirable to see how the formulas work in this situation. A natural tool is microlocal analysis. However, it seems that not all currently found inversion formulas can be conveniently analyzed from this point of view.

We find that the inversion formula by \cite{Kun07} for spherical $\mS$, together with its variation for other geometries, comes from a simple oscillatory integral. Therefore, it is suitable to be analyzed from microlocal analytic point of view. In this article we consider such oscillatory integral, which depends on the observation surface $\mS$ as a parameter. We show that for a general surface $\mS$, the oscillatory integral defines an operator $\mT$ which can be written down in the form $\mB \mP \mR$ (where $\mB$ and $\mP$ will be defined later). At the same time, $\mT$ is a good approximate of identity operator $I$. As a consequence, we obtain a good approximate inverse $\mB \mP$ of $\mR$. We then show that the approximation works particularly well when $\mS$ has some special geometries. Our presentation has two goals:
\begin{itemize}
\item[1)] To understand and predict the existence of inversion formula for $\mR$ under some special geometry of the observation surface. Our approach is of micro-local analytic nature, so it does not provide the proof that a formula exactly inverts $\mR$. However, when a formula behaves micro-locally very much like the identity, it is reasonable to expect that it may give the exact inversion. Using this idea, we predict that $\mB \mP$ is the exact inversion formula for all convex quadratic of surfaces.  The proof of this result is the topic of an up coming paper.
\item[2)] To understand how the formula works with the limited data problem. We emphasize that our goal is NOT to study the general problem of what can and cannot be reconstructed from the spherical Radon transform, which was very deeply analyzed in \cite{SUSAR}. Our goal, instead, is to see how our particular inversion formula works microlocally under the influence of the geometry of the observation surface $\mS$. This might help to understand the ability and limitation of this inversion formula. One of our conclusions is the inversion formula works best with the planar observation surface, in terms of constructing the singularities. 
\end{itemize}

The article is organized as follows. In Section \ref{S:Full}, we consider $\mS$ to be the boundary of a convex bounded domain $\Og$. We show that the above-mentioned operator $\mT$ (which is of the form $\mB \mP \mR$) satisfies $\mT = I +\mK$, where $\mK$ is pseudo-differential operator of order at most $-1$. We then go further to obtain an asymptotic expansion of $\mT$. When $\Og$ is an elliptical domain, we show that $\mK$ is an infinitely smoothing operator. We also show that the same result holds for a parabolic domain $\Og$. This is a good indication that $\mT$ provides the inversion formula for parabolic domain. In Section \ref{S:Lim}, we consider the partial data problem.  Applying the same approach as for the case of full data, we arrive to the analog $\mT_p$ of the operator $\mT$. We show that $\mT_p$ is a pseudo-differential operator and derive a simple formula for its principal symbol. As a consequence, we deduce that $\mT_p$ reconstructs all the ``visible'' singularities of $f$. We then show that the {\bf full symbol} of $\mT_p$ is equal to $1$ in a conic set. Therefore, $\mT$ reconstructs all the singularities of $f$ in that conic set without any distortion. Our approach work especially well when $\mS$ is a hyperplane. We consider this special case in Section \ref{S:Plane}. We shows that $\mT_p$ reconstructs all the ``visible'' singularities of $f$ without any distortions. This suggests that when working with partial data, the planar observation surface is an optimal geometry. 
\subsection{Notations and background knowledge}
For the later convenience, we now fix some notations. Let $\mO$ be a domain; we denote by $\mD(\mO)=C_0^\infty(\mO)$ the space of all smooth functions compactly supported inside $\mO$. The space $\mD'(\mO)$ is the dual of $\mD(\mO)$, i.e., the space of all distributions on $\mO$. Also, $\mE'(\mO)$ is the subspace of $\mD'(\mO)$ that contains all distributions compactly supported in $\mO$.

We denote by $S^m(\mO)$ (and $\Psi^m(\mO)$) the class of amplitudes (and pseudo-differential operators)  on $\mO$ of order at most $m$. The space $\Psi^{-\infty}(\mO) = \bigcap\limits_{m} \Psi^m(\mO)$ contains all infinitely smoothing operators. The reader is referred to \cite{TrPseu,shubin01pseudo,hormander71fourier} for definitions and basic properties of amplitudes and pseudo-differential operators. We will use the notation $\wf(f)$ for the wave front set of a function/distribution $f$. We refer the reader to \cite{louis00local} for an exposition in wave front set and its connections to spherical Radon transform.

\section{Convex hypersurface $\mS$}\label{S:Full}
Let $\mS$ be the boundary of a convex bounded domain $\Og$. We assume that $f$ is compactly supported in $\Og$. Then, the spherical Radon transform $\mR$ is a Fourier Integral Operator (FIO) of order $\frac{1-n}{2}$ (see, e.g., \cite{Pal-IPI}). Therefore, $\mR$ extends to a bounded operator  from $\mE'(\Og)$ to $\mD'(\mS \times \rN_+)$. From now on, we use the notation $\mR$ for this extended operator.

Let us introduce the operator $\mT$, which is defined by the following oscillatory integral:\begin{eqnarray*} \mT(f)(x) = \frac{1}{2\pi^n} \int\limits_{\mS} \int\limits_{\rN} \int\limits_{\rN^{n}} e^{i (|y-z|^2-|x-z|^2) \llg } \; |\llg|^{n-1} \; \left<z-x,\nu_z \right> \; f(y) \, dy \, d\llg \, d\sg(z),~ x \in \Og.\end{eqnarray*}
We now decompose $\mT$. For simplicity, we first assume that $f \in C_0^\infty (\Og)$. Then,
\begin{eqnarray*} \mT(f)(x) = \frac{1}{2\pi^n}\int\limits_{\mS} \left<z-x,\nu_z\right>  \int\limits_{\rN}  \int\limits_{\rN_+} e^{i (r^2 - |x-z|^2) \llg } \, |\llg|^{n-1} \, \mR(f)(z,r) \, dr \,d\llg \, d\sg(z).\end{eqnarray*}
Let $\mP: C_0^\infty(\rN_+) \to C^\infty(\rN_+)$ be the pseudo-differential operator defined by \begin{equation} \label{E:P0} \mP(h)(r) =  \int\limits_{\rN} \int\limits_{\rN_+} e^{i(\tau^2-r^2) \llg} \, |\llg|^{n-1} \, h(\tau) \, d\tau \, d\llg .\end{equation}
We obtain \begin{eqnarray*} \mT(f)(x) = \frac{1}{2 \pi^n} \int\limits_{\mS} \left<z-x,\nu_z\right> \, \mP \mR(f)(z,.)\big|_{r=|x-z|} \, d\sg(z).\end{eqnarray*}
Let $\mB:C^\infty(\mS \times \rN_+) \to C^\infty(\Og)$ be the back-projection type operator \begin{eqnarray*}\mB(g) = \frac{1}{2 \pi^n} \int\limits_{\mS} \left<z-x,\nu_z \right> \, g(z,|x-z|) \,d\sg(z).\end{eqnarray*}
We arrive to the following decomposition: \begin{equation} \label{E:dec} \mT(f) = \mB \mP \mR(f).\end{equation}
Since $\mR, \mB$ are FIOs (e.g., \cite{Pal-IPI}) and $\mP \in \Psi^{n-1}(\rN_+)$, they all extend continuously to the corresponding spaces of distributions. We will see later that $\mT$ is a pseudo-differential operator. Hence, it also extends continuously to $\mE'(\Og)$ and the above identity holds for all $f \in \mE'(\Og)$. 

Working out the formula of operator $\mP$ more explicitly, we obtain the following formula of $\mT$:
\begin{eqnarray*}\mT(f)(x) = \left \{ \begin{array}{l} \frac{(-1)^\frac{n-1}{2}}{2 \pi^{n-1}} \int\limits_{\mS} \left<z-x,\nu_z\right> \left[D^{n-1} r^{-1}\mR(f)(z,r) \right] \big|_{r=|z-x|} \; d\sg(z), \quad \mbox{for odd } n, \\[12 pt] \frac{(-1)^{\frac{n-2}{2}}}{\pi^n} \int\limits_{\mS} \left<z-x,\nu_z\right> \int\limits_{\rN_+} \frac{D^{n-1}\left[\tau^{-1}\mR(f)(z,\tau)\right]}{|x-z|^2-\tau^2} \; \tau \; d\tau \; d\sg(z),\quad \mbox{ for even } n, \end{array} \right.\end{eqnarray*}
where $$D = \frac{1}{2 r} \frac{d}{dr}.$$

When $\mS$ is a sphere or ellipse, $\mT$ is the inversion formula obtained Kunyansky \cite{Kun07} (for spherical $\mS$),  Natterer \cite{natterer2012photo} and Haltmeier \cite{Halt-Inv} (for elliptical $\mS$). The the factor $\left<z-x,\nu_z\right>$ in the formula of $\mT$ is very useful. It simplifies the symbol calculus of $\mT$ as shown in the following theorem:
\begin{theorem} \label{T:PDC} Assume that $\mS$ is the boundary of a convex bounded domain $\Og$. Then, $\mT$ is a pseudo-differential operator whose principal symbol is equal to $1$.\end{theorem}
\begin{proof} Let us recall the formula: $$\mT(f)(x) =\frac{1}{2\pi^n}\int\limits_{\mS} \int\limits_{\rN} \int\limits_{\rN^n} e^{i \left[|y-z|^2-|x-z|^2 \right]\llg} \; |\llg|^{n-1} \left<z-x,\nu_z\right> f(y) \; dy \; d\llg \; d\sg(z).$$
We observe the simple identity: \begin{equation}\label{E:Phase} |y-z|^2-|x-z|^2 = 2 \left<x-y,z-x \right> + |x-y|^2.\end{equation}
Therefore, the Schwartz kernel of $\mT$ is: \begin{eqnarray*} K(x,y) &=&\frac{1}{2 \pi^n} \int\limits_{\mS} \int\limits_{\rN_+} e^{i \left[\left<x-y,2 [z-x] \llg \right>+ |x-y|^2 \llg \right]} \; |\llg|^{n-1} \, \left<z-x,\nu_z\right> \; d\llg \; d\sg(z)  \\ &+& \frac{1}{2 \pi^n}\int\limits_{\mS} \int\limits_{\rN_-}  e^{i \left[\left<x-y,2 [z-x] \llg \right>+ |x-y|^2 \llg \right]} \; |\llg|^{n-1} \, \left<z-x,\nu_z\right>  \; d\llg \; d\sg(z) \\ &=& \sum\limits_{\pm}K_\pm.\end{eqnarray*}
We first consider $K_+$. Let us introduce the change of variables $$(z,\llg) \in \mS \times \rN_+ \longrightarrow \xi=2[z-x] \llg \in \rN^n.$$ (It resembles the change from the polar to cartesian coordinates.) Straight forward calculations show that: \begin{eqnarray*} d\xi = 2^n |\llg^{n-1} \left<z-x,\nu_z\right>| \; d\llg \; d\sg(z) = 2^n |\llg|^{n-1} \left<z-x,\nu_z\right> \; d\llg \; d\sg(z) .\end{eqnarray*}
The last equality holds since $\left<z-x,\nu_z\right> > 0$ (since $\Og$ is convex).
We, hence, obtain: \begin{eqnarray*} K_+(x,y) =\frac{1}{2 (2\pi)^n} \int\limits_{\rN^n}   e^{ i \left[\left<x-y,\xi \right> + |x-y|^2 \frac{|\xi| }{2 |x-z_+(x,\xi)|} \right] } \; d\xi.\end{eqnarray*}
Here, $z_+(x,\xi)$ is the intersection of $\mS$ with the ray $\{x+t \xi: t>0\}$. The point $z_+(x,\xi)$ is uniquely determined since $\Og$ is convex and $x \in \Og$. 

Similarly, we obtain: \begin{eqnarray*}K_-(x,y) =\frac{1}{2(2 \pi)^n} \int\limits_{\rN^n}   e^{ i \left[\left<x-y,\xi \right> - |x-y|^2 \frac{|\xi| }{2 |x-z_-(x,\xi)|} \right] } \; d\xi,\end{eqnarray*} where $z_-(x,\xi)$ is the intersection of $\mS$ with the ray $\{x+t\xi: t < 0\}$. 

Due to standard theory of FIO (e.g., \cite[Theorem 3.2.1]{Soggeb}), $K_\pm$ are kernels of pseudo-differential operators whose principal symbol is $\frac{1}{2}$.  Hence, $K$ is the kernel of a pseudo-differential operator whose principal symbol is $1$. This concludes the proof of the theorem.
\end{proof}

\begin{coro}\label{C:P} $\mT$ extends continuously to $\mE'(\Og)\to \mD'(\Og)$. Moreover, $\mT = I + \mK$, where $\mK$ is pseudo-differential operator of order $-1$.\end{coro}
Since the operators $\mB, \mP, \mR$ extend continuously to corresponding spaces of distributions, the identity $\mT(f) = \mB \mP \mR(f)$ holds true for all $f \in \mE'(\Og)$. Corollary \ref{C:P} simply says that $\mB \mP$ is a parametrix of $\mR$. In imaging applications, it is reasonable to consider $\mT(f) = \mB \mP \mR(f)$ as the approximate of the image $f$. Although the error $\mK(f)= \mB \mP \mR(f) -f$ might not be small, it is smoother than $f$. Therefore, $\mT(f)$ preserves the main part (top order) of singularities of $f$. For example, let us assume that $f$ has some jump singularities. Then, computing $\mB \mP \mR(f)$, one recovers the same jumps at the same locations. The Schwartz kernel of $\mK = \mT -I$ was explicitly found in \cite{natterer2012photo} and \cite{Halt-Inv}. However, our approach is more convenient from micro-local analytic point of view. Indeed, we can compute the full symbol of $\mT$:

\begin{theorem} Let $z_\pm(x,\xi)$ be defined as in the proof of Theorem \ref{T:PDC} and $$J_k(x,\xi) = \frac{ |\xi|^k}{|x-z_+(x,\xi)|^{k}} + (-1)^k \frac{ |\xi|^k}{|x-z_-(x,\xi)|^{k}}.$$ 
The the full symbol $\sg(x,\xi)$ of $\mT$ is given by: $$\sg(x,\xi) \sim  1+ \sum_{k=1}^\infty \frac{(-i)^k}{2^k k!} \;\Delta^k_\xi J_k(x,\xi), \quad \forall~ (x,\xi) \in T^* \Og \setminus 0.$$
Here, $\Delta_\xi$ is the Laplacian applying to the variable $\xi$ and $\Delta_\xi^k$ is the $k$-th power of $\Delta_\xi$. 
\end{theorem}
The above theorem follows from the following asymptotic behavior of $\mT$:
\begin{lemma}\label{L:Tn} Let $\mT_{k}$ be defined by the formula:
\begin{equation} \label{E:Tk} \mT_k(f)(x) = \frac{1}{2(2\pi)^n} \int\limits_{\rN^n} \int\limits_{\rN^n} e^{i \left<x-y,\xi \right>} \; \Delta_\xi^k J_k(x,\xi) \; f(y) \; dy \; d \xi,\quad x \in \Og.\end{equation}
Then, for all $N \in \nN$: $$\mT - \left[I+ \sum_{k=1}^N \frac{(-i)^k}{2^k k!} \mT_k\right] \in \Psi^{-(N+1)}(\Og).$$
\end{lemma}

The above lemma can be restated as:
\begin{equation} \label{E:Asymp} \mT \equiv I + \sum_{k=1}^\infty \frac{(-i)^k}{2^k k!}  \mT_k\quad \mbox{ (up to infinitely smoothing factor)}. \end{equation}
It is similar to \cite[Theorem 4]{beylkin84inversion}, which was stated for the generalized Radon transform.

\begin{proof}[Proof of Lemma \ref{L:Tn}] From the proof of Theorem \ref{T:PDC}, we obtain that $\mT= \mT_+ + \mT_-$, where:
 \begin{eqnarray*} \mT_+f(x) = \frac{1}{2 (2\pi)^n} \int\limits_{\rN^n} \int\limits_{\rN^n}   e^{ i \left[\left<x-y,\xi \right> + |x-y|^2 \frac{|\xi| }{2 |x-z_+(x,\xi)|} \right] } \, f(y) \,dy \, d\xi,\end{eqnarray*}
 and \begin{eqnarray*}\mT_-f(x) =\frac{1}{2(2 \pi)^n} \int\limits_{\rN^n}\int\limits_{\rN^n}  e^{ i \left[\left<x-y,\xi \right> - |x-y|^2 \frac{|\xi| }{2 |x-z_-(x,\xi)|} \right] } \, f(y) \,dy \, d\xi.\end{eqnarray*} 
We now analyze $\mT_+$. Using Taylor's formula, we obtain: \begin{eqnarray*} && e^{i |x-y|^2 \frac{|\xi|}{2|x-z_+(x,\xi)|}} = \sum_{k=0}^N \frac{i^k}{2^k k!} |x-y|^{2k} \frac{|\xi|^k}{|x-z_+(x,\xi)|^k}  \\&&+  \frac{i^{N+1}}{2^{N+1} N!} |x-y|^{2(N+1)} \frac{|\xi|^{N+1}}{|x-z_+(x,\xi)|^{N+1}} \int\limits_0^1 (1-t)^N e^{i t |x-y|^2 \frac{|\xi|}{2|x-z_+(x,\xi)|}} dt.\end{eqnarray*}
We arrive to \begin{eqnarray*} \mT_+f(x) &=& \frac{1}{2 (2\pi)^n} \sum_{k=0}^N \frac{i^k}{2^kk!}  \int\limits_{\rN^n}\int\limits_{\rN^n} \, |x-y|^{2k} \, e^{ i\left<x-y,\xi \right>} \, \frac{ |\xi|^k }{ |x-z_+(x,\xi)|^k}f(y) \, dy \, d\xi \\&+& A(f)(x).\end{eqnarray*} Here, $$A(f)(x) = \frac{i^{N+1}}{2(2 \pi)^n 2^{N+1}N!} \int\limits_0\limits^1 (1-t)^N A_t(f)(x) \, dt,$$ where $A_t(f)(x)$ is equal to:
\begin{eqnarray*} \int\limits_{\rN^n} \int\limits_{\rN^n} e^{i\left( \left<x-y,\xi \right> + t |x-y|^2 \frac{|\xi|}{2|x-z_+(x,\xi)|}\right)} |x-y|^{2(N+1)}  \; \frac{|\xi|^{N+1}}{|x-z_+(x,\xi)|^{N+1}} \, f(y) \, dy \, d\xi. \end{eqnarray*}
Since $ |x-y|^{2k} e^{ i\left<x-y,\xi \right>} = (-1)^k \Delta^k_\xi e^{ i\left<x-y,\xi \right>} $, we obtain \begin{eqnarray*} \mT_+f(x) &=& \frac{1}{2 (2\pi)^n} \sum_{k=0}^N \frac{(-i)^k}{2^kk!}  \int\limits_{\rN^n}\int\limits_{\rN^n} \Delta^k_\xi \big[e^{ i\left<x-y,\xi \right>} \big] \; \frac{ |\xi|^k }{ |x-z_+(x,\xi)|^k} \; f(y) \; dy \; d\xi \\&+& A(f)(x).\end{eqnarray*}
Taking integration by parts with respect to $\xi$, we obtain: \begin{eqnarray*} \mT_+f(x) &=& \frac{1}{2 (2\pi)^n} \sum_{k=0}^N \frac{(-i)^k}{2^kk!}  \int\limits_{\rN^n}\int\limits_{\rN^n}  e^{ i\left<x-y,\xi \right>} \; \Delta^k_\xi \left[\frac{ |\xi|^k }{ |x-z_+(x,\xi)|^k} \right] \; f(y) \; dy \; d\xi \\&+& A(f)(x).\end{eqnarray*}

For any $t \in \rN$, due to the special form of its phase function, $A_t$ is a pseudo-differential operator (see, e.g., \cite{shubin01pseudo}). Moreover, the amplitude function of $A_t$ is of class $S^{N+1}(\Og \times \Og)$ and vanishes up to order $2(N+1)$ on the diagonal $\Delta=\{(x,x): x \in \Og\}$. Due to the standard theory of FIO (e.g., \cite[Proposition 1.2.5]{Ho1}), we obtain $A_t \in \Psi^{-(N+1)}(\Og)$. Therefore, $A \in \Psi^{-(N+1)}(\Og)$.

Similarly, we obtain: \begin{eqnarray*} \mT_-f(x) &=& \frac{1}{2 (2\pi)^n} \sum_{k=0}^N \frac{i^k}{2^kk!}  \int\limits_{\rN^n}\int\limits_{\rN^n}  e^{ i\left<x-y,\xi \right>} \; \Delta^k_\xi \left[\frac{ |\xi|^k }{ |x-z_-(x,\xi)|^k} \right] \; f(y) \; dy \; d\xi \\&+& B(f)(x),\end{eqnarray*}
where $B\in \Psi^{-(N+1)}(\Og)$.

Adding up the formulas of $\mT_+$ and $\mT_-$, we obtain $$\mT= I + \sum_{k=1}^N \frac{(-i)^k}{2^k k!}  \mT_k + A+B.$$ This finishes the proof of the theorem.
\end{proof}

 For some special geometries of $\mS$, $\mT_k$ is very easy to deal with. Then, some nice properties of $\mT$ can be drawn from Lemma \ref{L:Tn}, or equivalently the identity (\ref{E:Asymp}).  As examples, we now consider the cases $\mS$ is an ellipse and elliptical parabola.

\subsection{Elliptical domain}
\begin{theorem} \label{T:SE} Assume that $\mS$ is an ellipse. Then, $\mK= \mT - I$ is an infinitely smoothing operator. That is, $\mT(f)-f \in C^\infty(\Og)$ for all $f \in \mE'(\Og)$. \end{theorem}

\begin{proof}[Proof of Theorem \ref{T:SE}] From Lemma \ref{L:Tn}, it suffices to prove that $\mT_k \equiv 0$ for all $k \geq 1$. Without loss of generality, we can assume that $\mS$ is defined by: $$\mS = \{z \in \rN^n: \sum_{i=1}^n\og_i^2 \; z_i^2 =1\},$$ where $\og_i's$ are some fixed positive numbers. For two vectors $x,y \in \rN^n$, we define the inner product 
$$\left<x,y\right> = \sum_{i=1}^n \og_i^2 \; x_i^2,$$ and the (scaled) norm:
$$\|x\|= \sqrt{\left<x, x\right>}.$$
We now analyze $$J_k(x,\xi) = \frac{ |\xi|^k}{|x-z_+(x,\xi)|^{k}} + (-1)^k \frac{ |\xi|^k}{|x-z_-(x,\xi)|^{k}}.$$ To this end, let us first compute the distances $|x-z_+(x,\xi)|$ and $|x-z_-(x,\xi)|$. We recall that $z_\pm(x,\xi) \in \mS$ and $z_\pm(x,\xi) = x + t_\pm \xi$, for some $t_+>0$ and $t_-<0$. To find $t_\pm$, we solve the equation (of the ellipse $\mS$): $$\|(x+t\xi)\|^2=1,$$ or $$\|\xi\|^2 \; t^2 + 2 \left<x,\xi\right> t + (\|x\|^2 -1) = 0.$$
We obtain: \begin{eqnarray*} t_\pm = \frac{- \left<x,\xi\right> \pm \sqrt{\Delta'}}{\|\xi\|^2},\end{eqnarray*} where $$\Delta' = (1-\|x\|^2)\|\xi\|^2 + \left<x,\xi\right>^2$$ is a (homogeneous) polynomial of degree $2$ in $\xi$.
Therefore, \begin{eqnarray*}\frac{|\xi|}{|x-z_\pm(x,\xi)|} = \frac{|\xi|}{|t_\pm \xi|} =  \frac{1}{|t_\pm|} = \frac{\pm \left<x,\xi\right>  + \sqrt{\Delta'}}{(1-\|x\|^2)}.\end{eqnarray*}
We arrive to the formula \begin{eqnarray*} J_k(x,\xi) &=&  \frac{ |\xi|^k}{|x-z_+(x,\xi)|^{k}} + (-1)^k \frac{ |\xi|^k}{|x-z_-(x,\xi)|^{k}} \\ &=& \frac{\big[\left<x,\xi\right>  + \sqrt{\Delta'} \big]^k + (-1)^k \big[- \left<x,\xi\right>  + \sqrt{\Delta'} \big]^k }{(1-\|x\|^2)^k}.\end{eqnarray*} 
It is easy to see that $J_k(x,\xi)$ is a polynomial in $\xi$ of degree at most $k$. Therefore, $\Delta^k_\xi \; J_k(x,\xi) =0$. From (\ref{E:Tk}), we obtain $\mT_k =0$ for all $k \geq 1$. Hence, the asymptotic formula (\ref{E:Asymp}) then finishes the proof.
\end{proof}

The above result is weaker than the identity $\mT(f) =f$, which was proved in \cite{Kun07,natterer2012photo,Halt-Inv}. However, our approach has some advantage when dealing with partial data problem, which we will consider later.

\subsection{Parabolic domain} An elliptical parabola is defined (up to translation), by the equation:
$$\mS = \Big\{z \in \rN^n: \sum_{i=1}^{n-1} \og_i^2 \; z_i^2 =z_n \Big\},$$ where $\og_i >0$ for $i=0,..,n-1$. Although it is not a closed convex surface, it almost encloses the convex domain: $$\Og = \Big\{z \in \rN^n: \sum_{i=1}^{n-1} \og_i^2 \; z_i^2 < z_n \Big\},$$ in the following sense: given $x \in \Og$, for all directions $\xi \in \rN^n$, except for the vertical ones, the line through $x$ along direction $\xi$ intersect $\mS$ at exactly two points, on the opposite sides of $x$. 

We can still define $\mT$ as the a pseudo-differential operator from $\mE'(\Og)$ to $\mD'(\Og)$ and the general framework presented above also applies. We now prove a similar result to Theorem \ref{T:SE}:

\begin{theorem} \label{T:SP} Assume that $\mS$ is an elliptical paraboloid. Then, $\mK= \mT - I$ is an infinitely smoothing operator. That is, $\mT(f)-f \in C^\infty(\Og)$ for all $f \in \mE'(\Og)$. \end{theorem}
\begin{proof}[Proof of Theorem \ref{T:SE}] From Lemma \ref{L:Tn}, it suffices to prove that $\mT_k = 0$ for all $k \in \nN$.

To simplify the writing, we introduce some notations. Let $x'=(x_1,..,x_{n-1}),\, y'=(y_1,..,y_{n-1}) \in \rN^{n-1}$, we define the inner-product $$\left<x',y'\right> = \sum_{i=1}^{n-1} \og_i^2 \; x_i \; y_i,$$
and the corresponding norm: $$\|x'\| = \sqrt{\left<x',x'\right>}. $$
Then, the equation of $\mS$ reads as:
$$\mS = \Big\{z: \|z'\|^2 =z_n\Big\},$$ where $z=(z',z_n)$. 

To analyze $$J_k(x,\xi) = \frac{ |\xi|^k}{|x-z_+(x,\xi)|^{k}} + (-1)^k \frac{ |\xi|^k}{|x-z_-(x,\xi)|^{k}},$$ we first compute the distances $|x-z_+(x,\xi)|$ and $|x-z_-(x,\xi)|$. We recall that $z_\pm(x,\xi) = x + t_\pm \xi \in \mS$, where $t_+>0$ and $t_-<0$. Therefore, $t_\pm$ are the solutions of the equation: $$\|x'+t\xi' \|^2=x_n + t \xi_n.$$
or $$ \|\xi'\|^2 \; t^2+ (2\left<x',\xi'\right> - \xi_n) \; t+ (\|x'\|^2 -x_n) = 0.$$
Since $(\|x'\|^2 -x_n) <0$, the above equation always has two solution for all $\xi' \neq 0$: \begin{eqnarray*} t_\pm = \frac{- (2\left<x',\xi'\right> - \xi_n) \pm \sqrt{\Delta}}{2 |\xi'|^2},\end{eqnarray*} where $$\Delta = (2\left<x',\xi'\right> - \xi_n)^2 - 4 \|\xi'\|^2 (x_n - \|x'\|^2)$$ is a (homogeneous) polynomial of degree $2$ in $\xi$. 

Noting that $|x - z_\pm(x,\xi)| = |t_\pm| \; |\xi|$, we obtain: \begin{eqnarray*}\frac{|\xi|}{|x-z_\pm(x,\xi)|} = \frac{1}{|t_\pm|} = \frac{\sqrt{\Delta} \pm (2\left<x',\xi'\right> - \xi_n) }{(x_n-\|x'\|^2) }.\end{eqnarray*}
We arrive to the formula \begin{eqnarray*} J_k(x,\xi) &=&  \frac{ |\xi|^k}{|x-z_+(x,\xi)|^{k}} + (-1)^k \frac{ |\xi|^k}{|x-z_-(x,\xi)|^{k}}  \\&=& \frac{\big[\sqrt{\Delta} + (2\left<x',\xi'\right> - \xi_n) \big]^k +(-1)^k \big[ \sqrt{\Delta} - (2\left<x',\xi'\right> - \xi_n) \big]^k}{(x_n -\|x'\|^2)^k}.\end{eqnarray*} It is straight forward to see that $J_k(x,.)$ is a polynomial of degree at most $k$. Therefore, $\Delta^k_\xi \; J_k(x,\xi) =0$. From (\ref{E:Tk}), we obtain $\mT_k =0$ for all $k \geq 1$. This finishes our proof.
\end{proof}
Following the proof of Theorems \ref{T:SE} and \ref{T:SP}, one can easily shows that the same result holds for any surface $\mS$ defined, up to translation and rotation, by the quadratic equation:
\begin{equation}\label{E:S}
\sum_{j=1}^m \og_j^2 \; x_j^2 = \sum_{j=m+1}^n \og_j x_j + \og_{n+1}. 
\end{equation}
for any fixed $m$ ($1 \leq m \leq n$) and $\og_i \geq 0$ such that $\og' = (\og_1,..,\og_m) \neq (0,..,0)$ and $\og'' = (\og_{m+1},..,\og_n) \neq (0,..,0)$. 

We, indeed, expect that $\mT(f)=f$ for any surface defined by (\ref{E:S}). However, we still do not have a proof for this stronger result. It will be the subject of study for an up coming paper. 

\section{Partial data problem}\label{S:Lim} 

In practical applications, one can only measure data on a proper subset of $\mS$ and in a finite time period. Therefore, the data can be modeled as $$g(z,r)=  \chi(z) \, \eta(r) \, \mR(f)(z, r).$$ Here, $\chi,\eta$ are the spatial and time cut-off functions, respectively. That is, there are bounded subsets $\Gamma_0$ and $\Gamma$ satisfying $\Gamma_0 \subset \bar{\Gamma}_0 \subset \Gamma \subset \mS$ such that \begin{eqnarray*} \chi(z) = \left\{\begin{array}{l}1,~z \in \Gamma_0, \\[6 pt] 0,~ z \in \mS \setminus \Gamma,\end{array}  \right.\end{eqnarray*} 
and there are $R>0$ and (small) $\eg>0$ such that: \begin{eqnarray*} \eta(r) = \left\{\begin{array}{l}1,~r \leq R, \\[6 pt] 0,~ r \geq R+\eg,\end{array} \right.\end{eqnarray*} 

The above conditions say that the (correct) data is only available on the domain $\Gamma_0$ and time interval $[0,R]$. 


Let us assume that $\mS$ is the boundary of a convex domain $\Og$. We now consider the modified operator $\mT_p(f)$ given by the formula:
\begin{eqnarray*} \frac{1}{2 \pi^n} \int\limits_{\mS} \int\limits_{\rN} \int\limits_{\rN^{n}} e^{i (|y-z|^2-|x-z|^2) \llg } \; \left<z-x,\nu_z\right> \; |\llg|^{n-1}\; \eta(|y-z|) \; \chi(z) \; f(y) \; dy \; d\llg \, d\sg(z).\end{eqnarray*}
Following the derivation of the identity (\ref{E:dec}) in Section \ref{S:Full}, we obtain the decomposition:
\begin{eqnarray*}\label{E:Decpartial} \mT_p(f) =  \mB \mP (g).\end{eqnarray*}
We mention that, when $\mS$ is an ellipse, the effect of $\mT_p$ is nothing but applying the inversion formula by \cite{Kun07,natterer2012photo,Halt-Inv} to partial data $g$. It is reasonable to consider $\mT_p$ as an approximate inversion formula for limited data case. We analyze the effect of this inversion procedure, from micro-local analytic point of view. 

Following the argument in the proof of Theorem \ref{T:PDC}, we obtain that the Schwartz kernel of $\mT_p$: $K_p = K_p^+ + K_p^-$, where
\begin{eqnarray*} K^+_p(x,y) =\frac{1}{2 (2 \pi)^n} \int\limits_{\rN^n}   e^{ i \left[\left<x-y,\xi \right> + |x-y|^2 \frac{|\xi| }{2 |x-z_+(x,\xi)|} \right]} \; \eta(|y-z_+(x,\xi)|) \; a_+(x,\xi) \; d\xi,\end{eqnarray*}
and \begin{eqnarray*}K_p^-(x,y) =\frac{1}{2 (2\pi)^n} \int\limits_{\rN^n}   e^{ i \left[\left<x-y,\xi \right> - |x-y|^2 \frac{|\xi| }{2 |x-z_-(x,\xi)|} \right]} \; \eta(|y-z_-(x,\xi)|) \; a_-(x,\xi) \; d\xi.\end{eqnarray*}
Here, as in the previous section, $z_\pm(x,\xi)$ are the intersections of $\mS$ with the rays $\{x+ t \xi: t> 0\}$ and $\{x+ t \xi: t<0\}$, respectively. The functions $a_\pm(x,\xi)$ are defined by $$a_\pm(x,\xi) = \chi(z_\pm(x,\xi)).$$
Using the same argument as that for Theorem \ref{T:PDC}, we obtain:
\begin{theorem} \label{L:Symbol} Let $\mS$ be the boundary of a convex domain $\Og$.
Then, $\mT_p: f \to \mB \mP (g)$ is a pseudo-differential operator whose principal symbol is: $$\sg_0(x,\xi) = \frac{1}{2}\Big[\eta(|x-z_+(x,\xi)|) \, a_+(x,\xi) + \eta(|x-z_-(x,\xi)|) \, a_-(x,\xi)\Big].$$
\end{theorem}
Let us introduce the conic sets: \begin{eqnarray*} \mA_+ &=& \{(x,\xi) \in \tT^*\Og \setminus 0: z_+ \in \Gamma_0 \mbox{ and } |x-z_+(x,\xi)| \leq R\} \\  \mA_- &=& \{(x,\xi) \in \tT^*\Og \setminus 0: z_+ \in \Gamma_0 \mbox{ and } |x-z_-(x,\xi)| \leq R\}.
\end{eqnarray*}
Here, $\tT^*\Og \setminus 0 = \Og \times (\rN^n \setminus 0)$ is the cotangent bundle of $\Og$ excluding the zero section. Let $\mA = \mA_+ \cup \mA_-$, then we can write \begin{eqnarray*} \mA = (\tT^*\Og \setminus 0) \bigcap \left[\bigcup\limits_{(z,r) \in \Gamma_0 \times (0,T)} \tN^*(S_r(z)) \right],\end{eqnarray*}
where $\tN^*(S_r(z))$ is the conormal bundle of the sphere $S_r(z)$. Following \cite{palamodov00reconstruction}, we call $\mA$ the \emph{visible} (or \emph{audible}) zone. It is described in \cite{louis00local} that all the singularities of $f$ in $\mA$ are ``visible" in the data $g$. Therefore, they \emph{should be} constructed stably. 

We observe that $a_\pm(z,\xi) \geq 0$ for all $(x,\xi) \in \tT^* \Og \setminus 0$. Moreover, for any  $(x,\xi) \in \mA$, either $a_+(x,\xi) = 1$ or $a_-(x,\xi) = 1$. Therefore, $\sg_0(x,\xi) > 0$ for all $(x,\xi) \in \mA$. Lemma \ref{L:Symbol} shows that $\mT_p$ reconstructs the singularities of $f$ in $\mA$ stably. That is any singularities of $f$ at $(x,\xi) \in \mA$ produces a corresponding singularity in $\mT(f)$ at the same location and direction $(x,\xi)$ (however, the resulted singularity may be different from the original one in terms of magnitude). The operator $\mT_p$ was, indeed, used in \cite{xu04reconstructions} to reconstruct the singularities in the visible zone $\mA$. Lemma \ref{L:Symbol}, thus, provides a rigorous justification for their method. 

If $(x,\xi) \in \mA_- \cap \mA_+$, then $\sg_0(x,\xi) =1$. Therefore, the main part (top order) of singularities of $f$ at any $(x,\xi) \in \mA_+ \cap \mA_-$ can be reconstructed exactly. Moreover, if the geometry of $\mS$ is special, we can prove more:

\begin{theorem} \label{T:Epart} Let $\mS$ be an ellipse or an elliptical parabola and $\sg(x,\xi)$ be the full symbol of $\mT_p$. Then $\sg(x,\xi) =1$ for all $(x,\xi) \in \mA_+ \cap \mA_-$.  
\end{theorem}
\begin{proof}
Similar to Lemma \ref{L:Tn}, we obtain the asymptotics expansion:
\begin{equation} \label{E:Asymp} \mT_p \equiv \mT_{0,p} + \sum_{k=1}^\infty \frac{(-i)^k}{2^k k!}  \mT_{k,p}\quad \mbox{ (up to infinitely smoothing factor)},\end{equation}
Here,
\begin{equation*} \mT_{k,p} (f)(x) = \frac{1}{2(2\pi)^n} \int\limits_{\rN^n} \int\limits_{\rN^n} e^{i \left<x-y,\xi \right>} \; \Delta_\xi^k \big[J_{k,p}(x,y,\xi)\big]\; f(y) \; dy \; d \xi,\quad x \in \Og,\end{equation*} where
\begin{eqnarray*} J_{k,p} (x,y,\xi) &=& \eta(|y-z_+(x,\xi)|) \; a_+(x,\xi) \; \frac{ |\xi|^k}{|x-z_+(x,\xi)|^{k}} \\ &+& (-1)^k \; \eta(|y-z_-(x,\xi)|) \; a_-(x,\xi) \; \frac{ |\xi|^k}{|x-z_-(x,\xi)|^{k}}.\end{eqnarray*}
For $(x,\xi) \in \mA_+ \cap \mA_-$ and $y$ close to $x$, we have $a_\pm(x,\xi) = 1$ and $\eta(|y-z_\pm(x,\xi)|)=1$; hence, $J_{k,p}(x,y,\xi) = J_k(x,\xi)$. As proven in Theorems \ref{T:SE} and \ref{T:SP}, $J_k(x,\xi)$ is a polynomial in $\xi$ of degree at most $k$. Therefore, for all $k \geq 1$, the symbol of $\mT_k$ is zero in $\mA_+ \cap \mA_-$. Moreover, the symbol of $\mT_0$ is equal to $1$ in $\mA_+ \cap \mA_-$. Therefore, the symbol of $\mT$ is equal to $1$ in $\mA_+ \cap \mA_-$.
\end{proof}
As a consequence of the above result, we obtain $\mT_p(f) - f \in C^\infty(\Og)$ for any $f \in \mE'(\Og)$ satisfying $\wf(f) \subset \mA_+ \cap \mA_-$. Therefore, $\mT_p(f)$ reconstructs all the singularities of $f$ in $\mA_+ \cap \mA_-$ without any distortions. We notice that $(x,\xi) \in \mA_+$ means the singularity at $(x,\xi)$ is observed in the direction $\xi$; and $(x,\xi) \in \mA_-$ means the singularity at $(x,\xi)$ is observe in the opposite direction $-\xi$. We, hence, conclude that: if the singularity at $(x,\xi)$ is observed at both directions $\pm \xi$, it is reconstructed perfectly.  This is an interesting property that only exists for some special geometries of $\mS$ (in the above theorem, we prove for elliptical and parabolic observation surface $\mS$). It would be interesting to find all such geometries. 

\section{Hyperplane}\label{S:Plane}
We now consider the case $\mS$ is a hyperplane. Without loss of generality, we assume that $$\mS=\{x \in \rN^n: x_n=0\}.$$ 
Let us denote $$\Og=\{x \in \rN^n: x_n>0\}.$$
Of course, $\mS$ does not encloses (or "almost enclose") $\Og$. Therefore, the general framework in Sections \ref{S:Full} and \ref{S:Lim} does not directly apply. However, $\mS$ has the following property: given $x$ in $\Og$, for most directions $\xi$, except for the horizontal ones, the line $\ell(x,\xi)$, which pass through $x$ along direction $\xi$, intersects $\mS$ at exactly one point. This, in some sense, means that $\mS$ "half-encloses" $\Og$. We, hence, modify the formula of $\mT$ by multiplying it by $2$. That is, 
\begin{eqnarray*} \mT(f)(x) = \frac{1}{\pi^n} \int\limits_{\mS} \int\limits_{\rN} \int\limits_{\rN^{n}} e^{i (|y-z|^2-|x-z|^2) \llg } \; \left<z-x,\nu_z\right> \; |\llg|^{n-1} \; f(y) \;dy \; d\llg \; d\sg(z).\end{eqnarray*}
Repeating the arguments in Section \ref{S:Full}, we obtain the following decomposition of $\mT$:
 \begin{eqnarray}\label{E:Dec2} \mT(f)(x) = \frac{x_n}{\pi^n} \; \mR^* \mP \mR(f)(x),\end{eqnarray} where $\mR^*$ is the formal adjoint of $\mR$. More explcitly, $\mT(f)$ can be written down in terms of spherical Radon transform $\mR$ as follows:
\begin{eqnarray*}  \mT(f)(x) &=& \left\{ \begin{array}{l} \frac{(-1)^{\frac{n-1}{2}}x_n}{\pi^{n-1}} \int\limits_{\mS} D^{n-1} \left[r^{-1}\mR(f) (z,r) \right] \big|_{r = |x-z|} \; d\sg(z), \mbox{ if n is odd}, \\ [6 pt]\frac{2(-1)^{\frac{n-2}{2}}x_n }{\pi^n}  \int\limits_{\mS} \int\limits_{\rN_+} \frac{D^{n-1}\left[\tau^{-1}\mR(f)(z,\tau)\right]}{|x-z|^2-\tau^2}\; \tau \; d\tau \; d\sg(z), \mbox{ if n is even}. \end{array} \right.\end{eqnarray*}

It is well known that $\mT(f) =f$ (see, e.g., \cite{BuKar,NaRa,XW05,Beltukov}). That is the above formula is an exact inversion of $\mR$. However, in practical applications such as SONAR (see, e.g. \cite{QuintoSONAR}), the partial data problem is of more importance.  We assume that the data is now modeled as $g(z,r) = \chi(z) \; \eta(r) \; \mR(f)(z,r)$, where $\eta$ and $\chi$ are defined as in Section \ref{S:Lim}. We then consider the corresponding operator $\mT_p(f)(x)$, which is given by:
\begin{eqnarray*} \frac{1}{\pi^n} \int\limits_{\mS} \int\limits_{\rN} \int\limits_{\rN^{n}} e^{i (|y-z|^2-|x-z|^2) \llg } \left<z-x,\nu_z\right> |\llg|^{n-1} \; \chi(z) \; \eta(|y-z|) \; f(y) \; dy \; d\llg \; d\sg(z).\end{eqnarray*}
We obtain:
\begin{eqnarray*} \mT_p(f)(x) =\frac{ x_n}{\pi^n} (\mR^* \mP (g)) (x).\end{eqnarray*}
This is just the application of the above inversion formula to the partial data $g$. 

We now analyze the micro-local properties of $\mT_p$. For any $(x,\xi) \in \tT^*\Og \setminus 0 \equiv \rN^n_+\times (\rN^n \setminus 0)$ such that $\xi$ is not parallel to the plane $\mS$, $\ell(x,\xi)$ intersects with $\mS$ at exactly one point $z(x,\xi)$. Let us define $$a(x,\xi) = \chi(z(x,\xi)).$$ The function $a(x,\xi)$ can be extended smoothly to $\tT^* \rN^n_+$ by zero.
\begin{theorem} \label{T:PrincipalPlane} The operator $\mT_p: f \to \frac{x_n}{\pi^n} \mP^*\mR(g)$ is a pseudo-differential operator whose principal symbol is $$\sg_0(x,\xi)= a(x,\xi) \; \eta(|x-z(x,\xi)|).$$
\end{theorem}
\begin{proof} 
We can rewrite the Schwartz kernel of $\mT_p$ as: \begin{eqnarray*} K_p(x,y) = \frac{1}{\pi^n} \int\limits_{\mS}  \int\limits_{\rN} e^{i \left[\left<x-y,2 [z-x] \llg \right>+ |x-y|^2 \llg \right]}\; \left<z-x,\nu_z\right> \; |\llg|^{n-1} \; \eta(|y-z|) \; \chi(z) \; d\llg \; d\sg(z) .\end{eqnarray*}
Let us make the following change of variables $$(z,\llg) \in \mS \times \rN \longrightarrow \xi(z,\llg) = 2[z-x] \llg \in \rN^n.$$ Straight forward calculations show that the Jacobian of the change is $2^n \left<z-x,\nu_z\right> |\llg|^{n-1}$.
Noting that $\llg = -\frac{\xi_n}{2x_n}$, we arrive to 
\begin{eqnarray}\label{E:Schp} K_p(x,y) = \frac{1}{(2\pi)^n}\int\limits_{\rN^n} e^{i \left<x-y,\xi \right>} \; e^{-i|x-y|^2 \frac{\xi_n}{2 x_n}} \; a(x,\xi) \; \eta(|y-z(x,\xi)|) \; d\xi.\end{eqnarray}
Standard theory of FIO (see, e.g.,  \cite[Theorem 3.2.1]{Soggeb}) concludes the proof of the theorem.
\end{proof}
As a corollary, $\mT_p$ extends continuously to $\mE'(\Og) \to \mD'(\Og)$. Moreover, its principal symbol $\sg_0(x,\xi)$ is equal to $1$ in the visible zone:
\begin{eqnarray*} \mA &=& \{(x,\xi) \in \tT^*\Og \setminus 0 : z(x,\xi) \in\Gamma_0, \mbox{ and } |x-z(x,\xi)|< R\}\\ &=& (\tT^*\Og \setminus 0) \bigcap \left(\bigcup_{(z,r) \in \Gamma_0 \times (0,R)} \tN^* (S_r(z)) \right).\end{eqnarray*}
Therefore, $\mT_p(f) = \frac{x_n}{\pi^n}\mR^*\, \mP (g)$ reconstructs ALL the visible singularities of $f$ without distorting the top order of singularities. For example, if the function $f$ has a jump singularity at $(x,\xi) \in \mA$ then $\mT_p(f) =\frac{x_n}{\pi^n} \, \mR^* \, \mP (g)$ also has a jump with the same magnitude at the same location. We actually can prove even a stronger result:
\begin{theorem}\label{T:Planar}
Let $f \in \mE'(\rN^n_+)$ such that $\wf(f) \subset \mA$, then $$\mT_p(f)(x) - f(x) \in C^\infty(\Og).$$
\end{theorem}
\begin{proof}
From (\ref{E:Schp}):
\begin{eqnarray*} \mT_p(f)(x) = \frac{1}{(2\pi)^n}\int\limits_{\rN^n}\int\limits_{\rN^n} e^{i \left<x-y,\xi \right>} e^{-i|x-y|^2 \frac{\xi_n}{2 x_n}} \; a(x,\xi)  \; \eta(|y-z(x,\xi)|) \; f(y) \; dy \; d\xi.\end{eqnarray*}
Using the Taylor's expansion for $e^{-i|x-y|^2 \frac{\xi_n}{2 x_n}}$ and arguing as in Lemma \ref{L:Tn}, we deduce the asymptotic expansion of $\mT_p$:
\begin{equation}\label{E:asym} \mT_p -\sum_{k=0}^\infty \frac{i^k}{(2|x_n|)^k k!}  \mT_{k,p} \in \Psi^{-\infty}(\Og).\end{equation}
Here, \begin{eqnarray*} \mT_{k,p}(f)(x) &=&\frac{(-1)^k}{(2 \pi)^n} \int\limits_{\rN^n} \int\limits_{\rN^n} |x-y|^{2k}\; e^{i \left<x-y,\xi \right>}   \; \xi_n^k \; a(x,\xi) \; \eta(|y-z(x,\xi)|) \; f(y) \; dy \; d\xi\\ &=&  \frac{1}{(2 \pi)^n} \int\limits_{\rN^n} \int\limits_{\rN^n} \Delta^k_\xi \; \left[e^{i \left<x-y,\xi \right>}\right] \; \eta(|y-z(x,\xi)|) \; a(x,\xi) \; \xi_n^k  \; f(y) \; dy \; d\xi.\end{eqnarray*}

Taking integration by parts with respect to $\xi$, we obtain:
\begin{eqnarray*} \mT_{k,p} (f)(x) = \frac{1}{(2\pi)^n} \int\limits_{\rN^n}\int\limits_{\rN^n} e^{i \left<x-y,\xi \right>} \; \Delta^k_\xi \left[a(x,\xi) \; \eta(|y-z(x,\xi)|) \; \xi_n^k  \right]  \; f(y) \; dy \; d\xi.\end{eqnarray*}
We notice that if $(x,\xi) \in \mA$ and $y$ is close to $x$, then $a(x,\xi) = 1$ and $\eta(|y-z(x,\xi)|) =1$. Hence, for $(x,\xi) \in \mA$ the full symbol if $\mT_k$ is 
\begin{eqnarray*}\sg^k(x,\xi) = \Delta_\xi^k \; \xi_n^k = \left\{\begin{array}{l} 1,~k=0,\\[6 pt] 0,~k \geq 1.\end{array} \right. \end{eqnarray*} From the asymptotic expansion (\ref{E:asym}), we obtain the full symbol $\sg(x,\xi)$ of $\mT$ satisfies $\sg(x,\xi)=1$ for all $(x,\xi) \in \mA$. Therefore, $\mT(f) - f \in \mD^\infty(\Og)$ if $\wf(f) \in \mA$.
\end{proof}

The above theorem shows that $\frac{x_n}{\pi^n}\mR^*\, \mP(g)$ reconstructs {\bf ALL} the ``visible'' singularities of $f$ up to infinite order (i.e, perfectly without any distortion). This is quite a unique property of planar observation surface $\mS$. It suggests that the planar observation surface is probably the best in terms of reconstructing ``visible'' singularities. 

\section*{Acknowledgment}
The author is thankful to Professor P. Stefanov for his critical comments to the preliminary version of the article.


\def\dbar{\leavevmode\hbox to 0pt{\hskip.2ex \accent"16\hss}d}

\end{document}